\documentclass[12pt,a4paper]{amsart}
\usepackage{graphicx,amssymb}
\usepackage{amsmath,amssymb,amscd}
\input xy
\xyoption{all}
\usepackage[all]{xy}
\usepackage{tikz}

\newcommand{\lra}{\longrightarrow}

\newcommand{\cU}{\mathcal{U}}
\newcommand{\cO}{\mathcal{O}}

\newcommand{\gr}{\text{gr}}
\newcommand{\rk}{\text{rk}}

\theoremstyle{plain}
\newtheorem{theorem}{Theorem}[section]
\newtheorem{lem}[theorem]{Lemma}
\newtheorem{prop}[theorem]{Proposition}
\newtheorem{cor}[theorem]{Corollary}
\newtheorem{conj}{Conjecture}
\newtheorem{rem}[theorem]{Remark}
\newtheorem{ex}[theorem]{Example}

\numberwithin{equation}{section}
\begin{document}
\title[Twisted BN loci]{New examples of twisted Brill-Noether loci II}

\author{L. Brambila-Paz}
\address{CIMAT\\Apdo. Postal 402\\ C.P. 36240\\ Guanajuato\\ Mexico}
\email{lebp@cimat.mx}
\author{P. E. Newstead \dag}

\date{\today}

\thanks{The authors are members of the research group VBAC (Vector Bundles on Algebraic Curves). The first author thanks Daniel Posada for his help in checking some of the calculations.}

\begin{abstract} Our purpose in this paper is to construct new examples of twisted Brill–Noether loci on curves of genus $g\geq 2$
with negative expected dimension. We begin by completing the proof of Butler’s conjecture for coherent systems of type
$(n,d,n+1)$, establishing the birationality, smoothness, and irreducibility of the corresponding loci. We also produce new points on the BN map.
\end{abstract}
\maketitle
\tableofcontents

\section{Introduction}\label{intro}

This paper constitutes Part II of \cite{bpn}. Our aim is to construct twisted Brill–Noether loci that differ from those considered in \cite{hhn} and \cite{bpn}. As in \cite{bpn}, many of the examples we construct will have negative Brill–Noether numbers. In \cite{bpn}, we constructed many of the examples by generalising the construction of \cite{bmno}, replacing tensoring with line bundles by tensoring with vector bundles of higher rank. We also used the dual span for sufficiently high degrees.

Our principal tool in this paper is the dual span construction in its general form. This allows us to apply results related to Butler’s Conjecture (see Conjecture \ref{conj1}), particularly in the case of line bundles. However, we also consider cases beyond those in which Butler’s Conjecture  is known to hold.

To state the results that provide the examples we seek, we begin with a brief review of the definitions from \cite{bpn} that we shall use throughout the article. Given a generated coherent system $(E,V)$ of type $(n,d,v)$, the dual span bundle $D_{E,V}$ is defined as the dual of the kernel of the evaluation map (see Subsection \ref{ss25}).
If $H^0(E^*)=0$, the coherent system $(D_{E,V}, V)$ is of type $(v-n,d,v)$. Our initial interest lies in the case
$v=n+1 $, which plays a central role in the construction of our examples.

In general, for any  smooth irreducible projective curve $C$ of genus $g\ge2$ and $n>1$ the condition that the Brill-Noether number satisfies
$\beta (n,d,k)\geq 0$ is neither necessary nor sufficient for the non-emptiness of the Brill–Noether locus
$ \widetilde{B}(n,d,k)$.
Nevertheless, for a general curve, we first establish the following result (see Corollary \ref{c4}):
$$ \widetilde{B}(n,d,n+1)\ne \emptyset \text{ {\it {if and only if}}} \  \beta(n,d,n+1) \geq 0. $$
$$ \text{\it{The same holds for the stable locus}} \  {B}(n,d,n+1)
 \ \text{\it{except when }} g=2 \  \text{{\it and}} \  d=2n.$$

Then, following many partial results (see, for example, \cite{bp, bbn1, bbn2}, and \cite{fla}), we prove the following
theorem, which completes the proof of Butler’s Conjecture for
$(n,d,n+1)$, establishing the smoothness and irreducibility of the corresponding loci.

In order to state the theorem, let $M(n,d)$ the moduli space of stable bundles of rank $n$ and degree $d$ on $C$ and $ S_0(n,d,v)$ denote the space of generated $\alpha $-stable coherent systems of type $(n,d,v)$, for sufficiently small $\alpha >0$, and let $S(n,d,v)$ denote:
\begin{multline}\label{eq601}
S(n,d,v)=\{(E,V)|E\in M(n,d),V\subset H^0(E) \text{ generates }E,\\ \dim V=v, D_{E,V} \text{ stable}\}.
\end{multline}

\begin{theorem}(Theorem \ref{t5})
Let $C$ be a general curve of genus $g\geq 3$,  and let $n$ and $d$ be positive integers. Then, for a general coherent system $(E,V)\in S_0(1,d,n+1)$, the dual span
lies in $S_0(n,d,n+1)$. Moreover, $ S(n,d,n+1)$  is an open dense subset of $S_0(n,d,n+1)$, and is isomorphic to an open subset of $S_0(1,d,n+1)$. Furthermore,
$S_0(1,d,n+1)$ and $S _0(n,d,n+1)$ are both smooth and irreducible of the expected dimension, and are birational.
\end{theorem}

In what follows, we shall usually abbreviate Brill-Noether to BN.

For $i=1,2$, let $M_i:=M(n_i,d_i)$ denote the moduli space of stable bundles of rank $n_i$ and degree $d_i$. Similarly, we write $\widetilde{M}_i$ for the corresponding moduli space of S-equivalence classes of semistable bundles. We define the \textit{universal twisted BN locus} as follows:
\begin{equation}\label{eq1}
B^k(\cU_1,\cU_2):=\{(E_1,E_2)\in M_1\times M_2\, |\,h^0(E_1\otimes E_2)\ge k\}.
\end{equation}The associated BN number (see \cite[(3.1)]{hhn} is now
\begin{equation}\label{eq2}
\beta^k(\cU_1,\cU_2):=\dim M_1+\dim M_2-k(k-n_2d_1-n_1d_2+n_1n_2(g-1)).
\end{equation}
We define similarly the corresponding semistable locus
\[\widetilde{B}^k(\cU_1,\cU_2):=\{([E_1],[E_2])\in \widetilde{M}_1\times\widetilde{M}_2|h^0(\gr E_1\otimes\gr E_2)\ge k\}.\]
We have symmetries $B^k(\cU_1,\cU_2)\cong B^k(\cU_2,\cU_1)$, and isomorphisms
\begin{itemize}\item $B^k(\cU_1,\cU_2)\cong B^{k_1}(\cU_1^*,\cU_2^*\otimes p_C^*K_C)$ and
\item $B^k(\cU_1,\cU_2)\cong B^k(\cU_1\otimes p_C^*L^*,\cU_2\otimes p_C^*L)$ for any line bundle $L$.
\end{itemize}

Let $E_2$ be any vector bundle on $C$ of rank $n_2$ and degree $d_2$.
The \textit{twisted BN locus} is defined as
 \begin{equation}\label{eq4}
B(n_1,d_1,k)(E_2):=\{E_1\in M_1\,|\,h^0(E_1\otimes E_2)\ge k\},
\end{equation}
with associated BN number
\begin{equation}\label{eq5}
\beta(n_1,d_1,k)(E_2):=\dim M_1-k(k-n_2d_1-n_1d_2+n_1n_2(g-1)).,
\end{equation}

We define also
\[\widetilde{B}(n_1,d_1,k)(E_2):=\{[E_1]\in \widetilde{M}_1\,|\,h^0(\operatorname{gr}E_1 \otimes E_2)\ge k\},\]
where $[E_1]$ denotes the S-equivalence class of $E_1$ and $\gr E_1$ is the corresponding graded bundle.

In the special case $E_2=\cO_C$, we obtain the standard (untwisted) higher rank BN locus $B(n_1,d_1,k)$, with expected dimension
\begin{equation}\label{eq44}
\beta(n_1,d_1,k):=n_1^2(g-1)+1-k(k-d_1+n_1(g-1)).
\end{equation} In particular, $B(1,d_1,k)$ is the classical BN locus $W^{k-1}_{d_1}$ and if $B^k(\cU_1,\cU_2)\ne \emptyset$ then $\widetilde{B}(n_1n_2,n_1d_2+n_2d_1,k)\ne \emptyset$.
 Moreover,  $B^k(\cU_1,\cU_2)\ne \emptyset$ if $B(n_1,d_1,k)(E_2)\ne \emptyset $ and $E_2$ is stable.

Many of the results in \cite{bpn}  are valid for any smooth curve. In this paper, we present some general results that are valid for arbitrary
$C$, but almost all the examples we construct require
$C$ to be a general curve. In certain cases, we can be more specific and allow
$C$ to be an arbitrary Petri curve.  Our objective is to determine the conditions under which the images of the morphisms
\begin{multline}\label{phi}
\Phi:B(n_1,d_1,k_1)\times S(n,d_2,n_2+n)\lra M(n_1,d_1)\times M(n_2,d_2)\\
(E_1,E)\longmapsto (E_1,D_{E,V}), \ \ \ \ \ \ \ \ \ \ \ \ \ \ \ \
\end{multline}
and
\begin{multline}\label{psi}
\Psi: B(n_1,d_1,k_1)\times  S(n,d_2,n+n_2)\lra M(n_1,d_1)\times M(n_2,-d_2)\\
(E_1,E)\longmapsto (E_1,D_{E,V}^*),  \ \ \ \ \ \ \ \ \ \ \ \ \ \ \ \
\end{multline}
lie in $B^k(\cU_1,\cU_2)$.
We begin by proving Lemma 2.6, which, as in Proposition 4.1, applies to various situations. Using these results, together with those on Butler's Conjecture, we establish conditions under which the images of the morphisms $\Psi$ and $\Phi$ are non-empty and contained in the corresponding locus $B^k(\cU_1,\cU_2)$. We construct explicit examples for which $B^k(\cU_1,\cU_2)$ is non-empty and $\beta^k(\cU_1,\cU_2)<0$.

 Some of the results derived from Lemma \ref{l4} are the following.

\begin{theorem}[Theorem \ref{t01}]\label{thm:main}
Let \( C \) be a smooth projective curve of genus \( g \ge 2 \).
Assume that \( B(n_1, d_1, k_1) \) \textup{(resp.} \( \widetilde{B}(n_1, d_1, k_1) \textup{)} \)
and \( S(n, d_2, n_2 + n) \) are non-empty.
Suppose further that
\[
k_2 = n_2 + n, \qquad n  d_1 < n_1 d_2.
\]
If \( k = k_1 k_2 \), then the image of the morphism
\begin{multline}
\Phi:B(n_1,d_1,k_1)\times S(n,d_2,n_2+n)\lra M(n_1,d_1)\times M(n_2,d_2)\\
(E_1,E)\mapsto (E_1,D_{E,V}) \ \ \
\end{multline}
 is contained in $B^{k}(\mathcal{U}_1, \mathcal{U}_2)$, and hence \[
B^{k}(\mathcal{U}_1, \mathcal{U}_2) \neq \emptyset
\quad \textup{(resp. }
\widetilde{B}^{k}(\mathcal{U}_1, \mathcal{U}_2) \neq \emptyset\textup{).}
\]
Moreover, \( \beta^{k}(\mathcal{U}_1, \mathcal{U}_2) < 0 \) if and only if
\begin{equation}\label{eq:beta-ineq}
\mu_1 + \mu_2
<
\lambda_1 \lambda_2
+
\left(
1 -
\frac{n_1^{2} + n_2^{2}}
{k_1\, n_1\, k_2\, n_2}
\right)(g - 1)
-
\frac{2}
{k_1\, n_1\, k_2\, n_2}.
\end{equation}
\end{theorem}

In particular, setting $n=1$ yields the following corollary.

\begin{cor}[Corollary \ref{tl}]\label{cor:main}
Let \( C \) be a general curve of genus \( g \ge 2 \).
Assume that \( \beta(1, d_2, n_2 + 1) \ge 0 \) and
\( B(n_1, d_1, k_1) \neq \emptyset \)
\textup{(resp. } \( \widetilde{B}(n_1, d_1, k_1) \neq \emptyset \textup{)}\).
If
\[
k_2 = n_2 + 1, \qquad d_1 < n_1 d_2 \ \  \mbox{and}  \qquad k = k_1 k_2,
\]
then
\[
B^{k}(\mathcal{U}_1, \mathcal{U}_2) \neq \emptyset
\quad \textup{(resp. }
\widetilde{B}^{k}(\mathcal{U}_1, \mathcal{U}_2) \neq \emptyset\textup{).}
\]
Moreover, \( \beta^{k}(\mathcal{U}_1, \mathcal{U}_2) < 0 \)
if and only if \begin{multline}\label{eq099}\mu_1+\mu_2<\left(1+\frac1{n_2}\right)\lambda_1+\left(1-\frac{n_1^2+n_2^2}{(n_1+1)n_1k_2n_2}\right)(g-1)\\
-\frac2{(n_1+1)n_1k_2n_2}.
\end{multline}
\end{cor}

As a direct consequence of Proposition  \ref{t2}, we obtain the following result (see Corollary \ref{c1} and Theorem \ref{t4}).

\begin{theorem} Let $C$ be a smooth curve of genus $g \geq 2$ with $n_1 \geq 2$. Suppose there exists $E_1 \in B(n_1,d_1,k_1)$ such that $h^1(E_1 \otimes E) = 0$ for $E \in S(n,d,v)$.
Suppose further that $$k \leq v k_1 - n d_1 - n_1 d + n n_1 (g-1).$$
Then,  for $(n_2,d_2) = (v-n,-d)$,
\begin{enumerate}
\item the image of the morphism
$$
\Psi : B(n_1,d_1,k_1) \times S(n,d,v) \longrightarrow M_1 \times M_2,
\quad (E_1,E) \longmapsto (E_1,D^*_{E,V})
$$is contained in $B^k(\mathcal{U}_1,\mathcal{U}_2)$, and in particular, $ B^k(\mathcal{U}_1,\mathcal{U}_2) \neq \varnothing. $
\item  For $d \gg 0$, $k_1 > n_1$ and any fixed values of  $n_1,d_1,k_1,n,e,f$ satisfying
 $$k=d(k_1-n_1)-e \ \mbox{and} \ v=d-n(g-1)-f,$$
\[
 n(g-1)(k_1-n_1)+nd_1+fk_1\le e<d(k_1-n_1),
\]
$ B^k(\mathcal{U}_1,\mathcal{U}_2) \neq \varnothing .$
Moreover, if
\[
d_1 < k_1 + n_1(g-1) - \frac{g-1}{k_1 - n_1},
\]
then $\beta^k(\mathcal{U}_1,\mathcal{U}_2) < 0.$
\end{enumerate}
\end{theorem}

The examples we are seeking (see Subsections \ref{ss42} and
\ref{ss52}) are constructed by considering bundles that either arise as of elements of \(S(r, d, n+r)\), for $r=1,2,3$, lie in the new Brill--Noether region (BPN region) introduced in \cite{bpn}, or are among the vector bundles defined in \cite{bmgno}.
Such examples exist even in the case \(n_{1} = n_{2} = 2\).

For convenience, we include the complete bibliography of \cite{bpn} with additional references appended at the end. There is one exception; \cite{bpn} now refers to Part I of this work rather than the current Part II.
Throughout the paper, $C$ will denote a smooth irreducible, projective curve of genus $g\ge2$ defined over ${\mathbb C}$. For a vector bundle $E$ on $C$ of rank $n$ and degree $d$, we denote by $\mu(E):=\frac{d}n$ the slope of $E$. We write also $h^0(E)$ for the dimension of the space of sections $H^0(E)$. The canonical line bundle on $C$ is denoted by $K_C$.

\section{Dual spans and non-emptiness of BN loci}\label{back}
In this section, we discuss the stability of dual span bundles, with particular emphasis on the case of rank $1$. We shall then give a comprehensive general result for the non-emptiness of the BN loci $B(n_1,d_1,n_1+1)$. Following many partial results (see, for example
\cite{bp,bbn1,bbn2}), the problem of the stability has now been completely solved for a general curve (see \cite{fla} and Corollary \ref{c4}).

\subsection{Kernel and dual span bundles}\label{ss25}
Let $(E,V)$ be a generated coherent system of type $(n,d,v)$; that is, $E$ is a vector bundle of rank $n$ and degree $d$ and  $V\subset H^0(E)$ a subspace of dimension $v$  that generates $E$. Consider the exact sequence:
\begin{equation}\label{eq6}
0\lra D_{E,V}^*\lra V\otimes\cO\lra E\lra0.
\end{equation}
The bundle $D_{E,V}^*$ may be referred to as a \textit{kernel bundle} (or \textit{syzygy bundle}) and has rank $v-n$ and degree $-d$. In particular, $D_{E,H^0(E)}^*$ is the bundle $D_E^*$ of \cite[(2.10)]{bpn}. Dualising \eqref{eq6}, we have
\begin{equation}\label{eq58}
0\lra E^*\lra V^*\otimes\cO\lra D_{E,V}\lra0.
\end{equation}
We call $D_{E,V}$ a \textit{dual span bundle}. If $H^0(E^*)=0$, then $h^0(D_{E,V})\ge v$; therefore $(D_{E,V},V^*)$ is a coherent system of type $(v-n,d,v)$.

We are concerned here with the stability of $D_{E,V}$. When $V=H^0(E)$, we have a general result in the case $d\ge2ng$ \cite[Theorem 2.5]{bpn}. In general, we have the conjecture of Butler, which is best expressed in terms of coherent systems. For the sake of accuracy, we give the conjecture in its original form, although we do not explain all the terms involved, as they will not be used in this paper. To state the conjecture, we denote by $S_0(n,d,v)$ the space of generated $\alpha$-stable coherent systems of type $(n,d,v)$ for small $\alpha>0$. For our purposes, the important facts are that, for any generated coherent system $(E,V)$ of type $(n,d,v)$,
\[E \text{ stable }\Longrightarrow (E,V)\in S_0(n,d,v)\Longrightarrow E \text{ semistable}.\]

\begin{conj}\label{conj1}\cite[Conjecture 2]{bu2}   Let $X$ be a general curve of genus $g\ge3$ and $n,d,v$ positive integers with $v>n$. Then, for a general $(E,V)\in S_{0}(n,d,v)$,
$(D_{E,V},V^*)\in S_{0}(v-n,d,v)$. Moreover,  $S_{0}(n,d,v)$ and $S_{0}(v-n,d,v)$ are birational.
\end{conj}

Some care is needed over the meaning of the term ``general'', since $S_0(n,d,v)$ could be reducible. To make sense of the last part of the conjecture, general has to mean \textit{general in any irreducible component}.

We are particularly interested in cases where $D_{E,V}$ is stable (resp. semistable). With this in mind, we make the following definitions.
\begin{multline}\label{eq60}
S(n,d,v):=\{(E,V)|E\in M(n,d),V\subset H^0(E) \text{ generates }E,\\ \dim V=v, D_{E,V} \text{ stable}\};
\end{multline}
\begin{multline}\label{eq61}
\widetilde{S}(n,d,v):=\{(E,V)|[E]\in \widetilde{M}(n,d),V\subset H^0(E) \text{ generates }E,\\ \dim V=v, D_{E,V} \text{ semistable}\}.
\end{multline}
The set $S(n,d,v)$ is an open subscheme of $S_0(n,d,v)$, so the non-emptiness of $S(n,d,v)$ can be seen as a form of Butler's Conjecture, although it is not exactly equivalent to Conjecture \ref{conj1}. Note that, if $S(n,d,v)\ne\emptyset$ (resp., $\widetilde{S}(n,d,v)\ne\emptyset$), then $v>n$. In fact, if $v=n+1$, $D_{E,V}$ is a line bundle, so necessarily stable. For interesting results on the stability of $D_{E,V}$, we can therefore assume that $v-n\ge 2$. If $\gcd(n,d)=1$, then $\widetilde{S}(n,d,v)=S(n,d,v)$.

From the definition of $S(n,d,v)$,  there is a morphism
 \begin{equation}\label{eq86}
 S(n,d,v)\lra B(v-n ,d,v):\ \ (E,V)\mapsto D_{E,V}.
 \end{equation}
This morphism is bijective if $d\ge 2ng$ and $V=H^0(E)$ (see \cite[Theorem 2.5]{bpn}).

\begin{rem}\label{r93}\begin{em}
The scheme $S(n,d,v)$ is not the same as $T(n,d,v)$ as defined in \cite{bmgno}, since $S(n,d,v)$ requires $D_{E,V}$ to be stable. This is the main reason why the non-emptiness of $S(n,d,v)$ is not equivalent to Conjecture \ref{conj1}.
However, if $\gcd(v-n,d)=1$, then $S(n,d,v)=T(n,d,v)$ and \cite[Theorem 3.9]{bmgno} can be stated using $S(n,d,v)$.
\end{em}\end{rem}

\subsection{Rank $1$}\label{ss26}

Unfortunately, very little is known about the conjecture when $n\ge2$, except in the case $V=H^0(E)$ and $d\ge 2ng$ \cite[Theorem 2.5]{bpn}. See also Example \ref{ex4} and  \cite{bmgno}  for some results in rank $2$, and \cite{chl} for further developments, including counterexamples and examples on bielliptic curves.
 However, when $n=1$, there are some strong results even in the case $V\ne H^0(E)$. The first is due to Mistretta.

\begin{theorem}\label{l1} \cite[Theorem 1.3]{mis} Let $C$ be a smooth curve of genus $g\ge2$. Suppose that $L$ is a line bundle of degree $d\ge2g+2c$ with $1\le c\le g$ and $V$ is a general subspace of $H^0(L)$ of codimension $c$. Then $D_{L,V}$ is semistable. Moreover, it is stable unless $d=2g+2c$ and $C$ is hyperelliptic.
\end{theorem}

For $C$ general, considerably more is known. In some cases, we can be more specific and allow $C$ to be an arbitrary Petri curve, that is, a curve for which the morphism
$H^0(L)\otimes H^0(K_C\otimes L^*)\lra H^0(K_C)$
is injective for every line bundle $L$.
For any Petri curve $C$ and any generated line bundle $L$, $D_L$ is semistable (in particular, this holds for general $L$ of degree $\ge g+1$) and precise conditions for stability are known \cite[Theorem 2]{bu2}. In fact, this is still true for any curve for which $B(1,d,v)=\emptyset$ whenever $\beta(1,d,v)<0$.

Coherent systems of type $(1,d,v)$  are classically referred to as \textit{linear systems of type} $(d,v)$.  On a Petri curve, the variety $G_{d}^{v-1}$ of linear systems of type $(d,v)$, if it is non-empty, has dimension $\beta(1,d,v)$ and is irreducible whenever $\beta(1,d,v)\ge1$. One can therefore speak of a general linear system $(L,V)\in G^{v-1}_d$. If $\beta(1,d,v)=0$, then all linear systems of type $(d,v)$ are regarded as general.

Butler’s Conjecture in rank $1$ can be stated as follows:
\smallskip

{\it if $C$ is general and $(L,V)$ is a general linear system on $C$ with $v\ge2$, then $V$ generates $L$ and $D_{L,V}$ is semistable \cite[Conjecture 1.1]{bbn2}.}
\smallskip

 This has been proved \cite[Theorem 5.1]{bbn2}. It is further conjectured that, if $C$ is Petri, then $D_{L,V}$ is always stable \cite[Conjecture 9.5]{bbn1}. Very recently, this conjecture was proved for a general curve by Farkas and Eric Larson \cite[Theorem 1.3]{fla}.

\begin{theorem}\label{l3}
Let $C$ be a general curve of genus $g\ge2$ and suppose that $n_1\ge2$. Let $(L,V)$ be a general linear system of type $(d_1,n_1+1)$. Then $V$ generates $L$ and $\beta(1,d_1,n_1+1)\ge0$, or, equivalently,
\begin{equation}\label{eq41}d_1\ge n_1+\frac{n_1g}{n_1+1}.
\end{equation}
Moreover, if $\beta(1,d_1,n_1+1)\ge0$, then $(L,V)$ exists and $D_{L,V}$ is stable except when $g=2$ and $d_1=2n_1$.
\end{theorem}
\begin{proof}
The fact that $V$ generates $L$, the inequality $\beta(1,d_1,n_1+1)\ge0$ and the existence statement follow from classical BN theory. The stability of $D_{L,V}$ is \cite[Theorem 1.3]{fla} when $n_1\ge3$. For $n_1=2$, see \cite[Theorem 7.1]{bbn1}.
\end{proof}

\begin{rem}\label{r8}\begin{em}
For an arbitrary Petri curve of genus $g\ge2$, we can still conclude that $D_{L,V}$ is stable for general $(L,V)$ in many cases. The most complete list known currently is as follows:
\begin{itemize}
\item[(a)] $n_1\le4$;
\item[(b)] $g\ge 2n_1-4$;
\item[(c)] there exists a generated linear system $(L_0,V_0)$ of type $(d_0,n_1+1)$ with $D_{L_0,V_0}$ stable and $d_1=d_0+an_1$ for some positive integer $a$;
\item[(d)] there exists a generated linear system $(L_0,V_0)$ of type $(an_1,n_1+1)$ for some integer $a$ with $D_{L_0,V_0}$ stable and $d_1>an_1$, $d_1\equiv\pm1\bmod n_1$;
\item[(e)] $d_1\le2n_1$;
\item[(f)] $d_1< n_1+g+\frac{(n_1^2-n_1-2)g}{2(n_1-1)^2}$ ($\le$ if $g$ is odd);
\item[(g)] $d_1\ge n_1 \left\lceil\frac{g+3}2\right\rceil+1$; $d_1\ge n_1\frac{g+2}2+1$ for $g$ even, $n_1\le\frac{g!}{\frac{g}2!(\frac{g}2)+1)!}$.
\end{itemize}
For the various cases, see \cite[Theorems 7.1-7.3 and 8.2, Propositions 6.6, 6.8]{bbn1} and \cite[Theorems 4.7 and 6.1]{bbn2}, \cite[Corollary 3.1]{bht} \cite[Th\'{e}or\`{e}me 2-B-1]{m1} and \cite[Proposition 2]{m3}.
\end{em}\end{rem}

\begin{rem}\begin{em}\label{r24}
There are results on the stability of $D_L$ and $D_{L,V}$ for special curves, and, in  particular, for curves of given Clifford index (see \cite{cam,bo,ms}). However, one has to assume (or prove) that $L$ is generated; it is no longer automatic that $B(1,d_1,n_1+1)$ contains a generated line bundle.
\end{em}\end{rem}

\subsection{Non-emptiness of $B(n_1,d_1,n_1+1)$}\label{ss23} It is clear that $B(n_1,d_1,n_1+1)\ne\emptyset$ if Theorem \ref{l3} or Remark \ref{r8} applies. It could be, however, that there are other cases in which $B(n_1,d_1,n_1+1)\ne\emptyset$. The following lemma shows that \eqref{eq41} is a necessary condition for non-emptiness. For $C$ general and $B(n_1,d_1,n_1+1)$, the lemma is included in \cite[Theorem 3.9(i)]{bp}; the proof is on the same lines as that of \cite[Proposition 3.8]{bp}.

\begin{lem}\label{l4} Let $C$ be a Petri curve and $\widetilde{B}(n_1,d_1,n_1+1)\ne\emptyset$. Then $\beta(1,d_1,n_1+1)\ge0$.
\end{lem}
\begin{proof} Let $F$ be a semistable bundle such that $[F]\in \widetilde{B}(n_1,d_1,n_1+1)$, and let $G$ be the subsheaf of $F$ generated by $H^0(F)$. It is easy to see that $G\cong G'\oplus{\cO}_C^s$, where $s=h^0(G^*)$ and $G'$ is a generated vector bundle with $h^0(G'^*)=0$. If $m=\rk G'$, then $h^0(G')\ge m+1$ and we can choose a subspace $V\subset H^0(G')$ of dimension $m+1$ which generates $G'$. We therefore have an exact sequence
\[0\lra L^*\lra V\otimes{\cO}_C\lra G'\lra0,\]
where $L=\det G'$. Since $C$ is Petri, it follows by classical BN theory that $\deg(L)\ge m+\frac{mg}{m+1}$. Since $\deg(G)=\deg(L)$ and $F$ is semistable, we have
\[1+\frac{g}{n_1+1}\le1+\frac{g}{m+1}\le\mu(G)\le\mu(F).\]
This gives \eqref{eq41}.
\end{proof}

\begin{cor}\label{c4}
Let $C$ be a general curve of genus $g\ge2$. Then the BN locus $\widetilde{B}(n_1,d_1,n_1+1)\ne\emptyset$ if and only if $\beta(n_1,d_1,n_1+1)\ge0$. The same holds for $B(n_1,d_1,n_1+1)$ except when $g=2$ and $d_1=2n_1$.
\end{cor}
\begin{proof}
Note that $\beta(n_1,d_1,n_1+1)=\beta(1,d_1,n_1+1)$. The result now follows from Lemma \ref{l4}  and Theorem \ref{l3}.
\end{proof}

This result is significant because, in general, the condition $\beta(n_1,d_1,k_1)\ge0$ is neither necessary nor sufficient for the non-emptiness of $B(n_1,d_1,k_1)$.

\begin{rem}\label{r22}\begin{em} Let $C$ be a Petri curve of genus $g\ge2$. Then $B(n_1,d_1,n_1+1)\ne\emptyset$ in all the cases indicated in Remark \ref{r8}. Note that $$\beta(1,d_1,n_1+1)=\beta(n_1,d_1,n_1+1)=1$$ (resp. $\beta(n_1,d_1,n_1+1)= 0$) if and only if $d_1=n_1+\frac{n_1g}{n_1+1}+\frac{1}{n_1+1}$ (resp. $d_1=n_1+\frac{n_1}{n_1+1})$. For example, for $g=11$, $\beta (2,13,5)=\beta(9,19,10)=1$ and $\beta (10,20,11)=0$. Moreover, $\beta(n_1,d_1,n_1+a)=1$ if and only if $d_1=n_1+\frac{an_1g}{n_1+a}+\frac{a^2}{n_1+a}$. For example, for $a=2$ and any genus $\beta (2,g+3,4)=1$.  Moreover, for a general curve of odd genus \(g \ge 5\) by \cite[Theorem~5.2(ii), (iii)]{bmgno}, $B(2,g+3,4)$ is non-empty.
\end{em}\end{rem}

As a further consequence of Theorem \ref{l3} and Lemma \ref{l4}, we can prove the original version of Butler's Conjecture (Conjecture \ref{conj1}) for $v=n+1$. The case $n+1$  was previously established in \cite[Theorem 5.1]{bbn2}, but without the birationality property; here, we additionally include smoothness and irreducibility in our statement.

\begin{theorem}\label{t5}
Conjecture \ref{conj1} holds in the case $n=1$. In fact, under the hypotheses of the conjecture, $S(1,d,v)$ is an open dense subset of $S_0(1,d,v)$, and is isomorphic to an open subset of $S_0(v-1,d,v)$. Moreover, both $S_0(1,d,v)$ and $S_0(v-1,d,v)$  are smooth and irreducible of the expected dimension, and they are birational.
\end{theorem}
\begin{proof}
The locus $S_0(1,d,v)$ is smooth and irreducible of the expected dimension $\beta(1,d,v)$ by classical BN theory. It is clear that $S(1,d,v)$ is open in $S_0(1,d,v)$ and it is non-empty whenever $S_0(1,d,v)\ne\emptyset$ by Theorem \ref{l3} and Lemma \ref{l4}. There is a natural morphism
\begin{equation}\label{eq87}
S(1,d,v)\lra S_0(v-1,d,v):\ \ (L,V)\mapsto (D_{L,V},V^*),
\end{equation}
which is clearly injective. Moreover, the restriction of the morphism
\[S_0(v-1,d,v)\lra S_0(1,d,v):\ \ (E,W)\mapsto (D_{E,W},W^*)\]
to those $(E,W)$ for which $E$ is stable provides an inverse to \eqref{eq87}. Since every $(E,W)\in S_0(v-1,d,v)$ is of the form  $(D_{L,V},V^*)$ for some $(L,V)\in S_0(1,d,v)$, it follows that $S_0(v-1,d,v)$ is irreducible. Birationality now follows from \eqref{eq87}. For smoothness of the expected dimension, see \cite[Lemma 4.2]{bp}.
\end{proof}

\begin{rem}\label{r87}\begin{em}
The proof remains valid for $g=2$. Although in the exceptional case $d_1=2n_1$, $D_{L,V}$ cannot be stable, it is semistable, and it follows that $(D_{L,V},V^*)\in \widetilde{S}_0(v-n,d,v)$.
\end{em}\end{rem}

\section{Examples using dual spans}\label{first}
In this section, we start with the same framework as in  \cite[Section 4]{bpn} and extend the construction given in \cite[Proposition 4.1]{bpn}. We are concerned with a situation where we have two vector bundles \(E_i\)  satisfying \(h^0(E_i)\ge k_i\), and we wish to explore when
$ h^0(E_1\otimes E_2)\ge k_1k_2.$
This essentially generalises the setting considered in \cite{bmno}.
For the examples we seek, we shall take $E_2$ to be a dual span bundle.

\subsection{The construction}\label{ss41}
Let $(E,V)$ be a generated coherent system of type $(n,d,v)$ and let $E_1$ be a bundle of rank $n_1$ and degree $d_1$ with $h^0(E_1)\ge k_1$.
Tensoring \eqref{eq58} with $E_1$ and taking global sections, we obtain the following exact sequence:
\begin{equation}\label{eq588}
0\lra H^0(E^*\otimes E_1)\lra V^*\otimes H^0(E_1)\lra H^0(D_{E,V}\otimes E_1)\lra \dots .
\end{equation}

The following lemma involves the dual span bundle and summarises a number of simple but useful facts. No assumptions of generality or stability are made, except where explicitly stated. As a result, the lemma applies in many situations.
\begin{lem}\label{p3}
Let $C$ be a smooth curve of genus $g\ge2$. As above, let $(E,V)$ be a generated coherent system of type $(n,d,v)$ and $E_1$ is a bundle of rank $n_1$ and degree $d_1$ with $h^0(E_1)\ge k_1$. Let $(n_2,d_2,k_2):=(v-n,d,v)$ and $k:=k_1k_2$. Suppose further that $h^0(E^*\otimes E_1)=0$.  Then $h^0(D_{E,V}\otimes E_1)\ge k$.

In particular,
\begin{itemize}
\item[(i)] if $D_{E,V}$ is stable (resp., semistable), then $B(n_2,d_2,k)(E_1)\ne\emptyset$ (resp., $\widetilde{B}(n_2,d_2,k)(E_1)\ne\emptyset$);
\item[(ii)] if $E_1$ is stable (resp., semistable), then $B(n_1,d_1,k)(D_{E,V})\ne\emptyset$ (resp., $\widetilde{B}(n_1,d_1,k)(D_{E,V})\ne\emptyset$);
\item[(iii)] if $D_{E,V}$ is stable (resp., semistable) and $B(n_1,d_1,k_1)\ne\emptyset$ (resp., $\widetilde{B}(n_1,d_1, k_1)\ne\emptyset$), then $B^k(\cU_1,\cU_2)\ne\emptyset$ (resp., $\widetilde{B}^k(\cU_1,\cU_2)\ne\emptyset$).
\end{itemize}

Moreover, writing $\mu_i=d_i/n_i$ and $\lambda_i=k_i/n_i$,
\begin{itemize}
\item[(iv)] the BN number $\beta^k(\cU_1,\cU_2)<0$ if and only if
\begin{equation}\label{eq80}\mu_1+\mu_2<\lambda_1\lambda_2+\left(1-\frac{n_1^2+n_2^2}{k_1n_1k_2n_2}\right)(g-1)-\frac2{k_1n_1k_2n_2}.
\end{equation}
\item[(v)] $\dim B^k(\cU_1,\cU_2)>\beta^k(\cU_1,\cU_2)$ if the morphism \eqref{eq86} is generically injective and
\begin{multline}\label{eq85}
\left(1-\frac1{k_2n_2}\right)\mu_1+\left(1-\frac1{k_1n_1}\right)\mu_2\\<\lambda_1\lambda_2-\frac{k_1^2+k_2^2}{k_1n_1k_2n_2}+\left(1-\frac{k_1n_1+k_2n_2}{k_1n_1k_2n_2}\right)(g-1).
\end{multline}
\end{itemize}
\end{lem}
\begin{proof}
The first statement follows directly from the sequence \ref{eq588}. Statements (i), (ii) and (iii) follow from the definitions.

(iv) If $k=k_1k_2$, by \eqref{eq2}, we have,
\begin{equation*}\label{eq88}
\beta^k(\cU_1,\cU_2)=(n_1^2+n_2^2)(g-1)+2-k_1k_2(k_1k_2-(n_2d_1+n_1d_2)+n_1n_2(g-1)).
\end{equation*}
Dividing by $k_1n_1k_2n_2$ and rearranging, we obtain \eqref{eq80}.

(v) Since \eqref{eq86} is generically injective, we have
\begin{equation}\label{eq98}
\dim B^k(\cU_1,\cU_2)\ge\beta(n,d,v)+\beta(n_1,d_1,k_1).
\end{equation}
Now
\[\beta(n,d,v)=\beta(v-n,d,v)=\beta(n_2,d_2,k_2)=\dim M_2-k_2(k_2-d_2+n_2(g-1)),\]
while $\beta(n_1,d_1,k_1)$ is given by a similar expression and $\beta^k(\cU_1,\cU_2)$ is given by \eqref{eq2}.
So, by \eqref{eq98},
\begin{multline*}\dim B^k(\cU_1,\cU_2)-\beta^k(\cU_1,\cU_2)\ge k(k-n_2d_1-n_1d_2+n_1n_2(g-1))\\-k_1(k_1-d_1+n_1(g-1))-k_2(k_2-d_2+n_2(g-1)).\end{multline*}

Since $k=k_1k_2$, the right hand side of this inequality equates to
\begin{multline}\label{eq90}
k_1^2k_2^2-k_1^2-k_2^2-k_1k_2(n_2d_1+n_1d_2-n_1n_2(g-1))\\+k_1(d_1-n_1(g-1))+k_2(d_2-n_2(g-1))
\end{multline}
Again, dividing by $k_1n_1k_2n_2$ and rearranging, we see that \eqref{eq90} is positive if and only if \eqref{eq85} holds.
\end{proof}

\begin{rem}\begin{em}\label{r84}
(i) Note that $\mu_2=\frac{d}{v-n}$ and this can be $>2$; this differs from the situation in \cite[Section 4]{bpn}. Note also that $\lambda_2=\frac{v}{v-n}>1$. Lemma \ref{p3} is valid for all positive integers $n_1$, $n_2$, but yields new examples only when $n_i\ge2$.

(ii) We have already remarked that \eqref{eq86} is bijective if $d\ge2ng$ and $V=H^0(E)$. If $C$ is general, $n=1$ and $0\le\beta(1,d,v)\le g$, then, by classical BN theory, the general element $L\in B(1,d,v)$ has $h^0(L)=v$. Since $\det D_{L,V}\cong L$, it follows from Theorem \ref{t5} that, in this case, \eqref{eq86} is birational and hence certainly generically injective.
\end{em}\end{rem}

\begin{cor}\label{c6}
Let $C$ be a smooth curve of genus $g\ge2$ and suppose that $(E,V)$ and $E_1$ are as in Lemma \ref{p3}. Let $(n_2,d_2,k_2)=(r(v-n),rd,rv)$ for some positive integer $r$ and let $k=k_1k_2$. If $D_{E,V}$ is semistable, then $\widetilde{B}(n_2,d_2,k)(E_1)\ne\emptyset$. If, in addition, $E_2$ is semistable, then $\widetilde{B}^k(\cU_1,\cU_2)\ne\emptyset$.
\end{cor}

\begin{proof}
Since $D_{E,V}$ is semistable, so is $D_{E,V}^{\oplus r}$. Moreover, $h^0(D_{E,V}^{\oplus r})\ge rv=k_2$. It follows from Lemma \ref{p3} that $h^0(D_{E,V}^{\oplus r}\otimes E_1)\ge k_1k_2=k$. The result now follows from Lemma \ref{p3}(i) and (iii).
\end{proof}

The following corollary may  be compared with \cite[Theorem 1.1]{hhn}.

\begin{cor}\label{c7}
Let $C$ be a smooth curve of genus $g\ge2$ and suppose that, for $1\le i\le r$, $(E^i,V^i)$ are generated coherent systems of type $(n,d,v)$ with $d$ a multiple of $v-n$ and the bundles $D_{E^i,V^i}$ pairwise non-isomorphic. Let $(n_2,d_2,k_2)=(r(v-n),rd+1,rv)$ for some positive integer $r$ and let $k=k_1k_2$. Suppose further that $E_1$ is a bundle of rank $n_1$ and degree $d_1$ with $h^0(E_1)\ge k_1$ and $h^0(E^{i*}\otimes E_1)=0$. Let $k=k_1k_2$. If $D_{E^i,V^i}$ are all stable, then $B(n_2,d_2,k)(E_1)\ne\emptyset$. If, in addition, $E_1$ is stable, then $B^k(\cU_1,\cU_2)\ne\emptyset$.
\end{cor}

\begin{proof}
Let $p\in C$ be a torsion sheaf of length $1$ and consider extensions
\[0\lra D_{E^1,V^1}\oplus\cdots\oplus D_{E^r,V^r}\lra E_2\lra{\mathcal O}_p\lra0.\]
These are classified by sequences $e^1,\ldots,e^r$ with $e^i$ in the fibre $(D_{E^i,V^i})_p$. If $d$ is a multiple of $\rk D_{E^i,V^i}=v-n$ and all $e^i\ne0$, it is easy to check that $E_2$ is stable. Since $h^0(D_{E^i,V^i}\otimes E_1)\ge vk_2$, it follows that $h^0(E_1\otimes E_2)\ge k_1k_2=k$. So $B^k(\cU_1,\cU_2)\ne\emptyset$ if $E_1$ is stable.
\end{proof}

\begin{theorem}\label{t9}
Suppose that $C$ is a general curve of genus $g\ge2$. Suppose further that $d>ng$ and that $n<v\le d-n(g-1)$. Let $(n_2,d_2,k_2)=(v-n,d,v)$, $k=k_1k_2$, $\mu_1=\frac{d_1}{n_1}$ and $\mu_1<\mu_2$. If $B(n_1,d_1,k_1)\ne\emptyset$, then $B(n_2,d_2,k_2)(F)\ne\emptyset$ for some $F$ of rank $n_1$ and degree $d_1$.
\end{theorem}

\begin{proof} By \cite[Theorem 3.3]{bmgno}, there exists a generated stable bundle $E$ of rank $n$ and degree $d$. It follows that, for $n<v<d-n(g-1)$, the general subspace $V$ of dimension $v$ of $H^0(E)$ generates $E$. Hence $D_{E,V}$ exists. Since $\mu_1<\mu_2$, $h^0(E^*\otimes E_1)=0$ for all $E_1\in B(n_1,d_1,k_1)$. The result follows from Lemma \ref{p3}(ii).
\end{proof}

Note that we make no assumption about the stability of $D_{E,V}$ here; this allows a much wider range of values for $n$ and $d$. However, under the additional assumption that the bundles are stable, we obtain the following theorem.

\begin{theorem}\label{t01}
Let \( C \) be a smooth projective curve of genus \( g \ge 2 \).
Assume that \( B(n_1, d_1, k_1) \) \textup{(resp.} \( \widetilde{B}(n_1, d_1, k_1) \textup{)} \)
and \( S(n, d_2, n + n_2) \) are non-empty.
Suppose further that
\[
k_2 = n + n_2, \qquad n  d_1 < n_1 d_2.
\]
If \( k = k_1 k_2 \), then the image of the morphism
\begin{multline}
\Phi:B(n_1,d_1,k_1)\times S(n,d_2,n_2+n)\lra M(n_1,d_1)\times M(n_2,d_2)\\
(E_1,E)\mapsto (E_1,D_{E,V})
\end{multline}
 is contained in $B^{k}(\mathcal{U}_1, \mathcal{U}_2)$, and hence \[
B^{k}(\mathcal{U}_1, \mathcal{U}_2) \neq \emptyset
\quad \textup{(resp. }
\widetilde{B}^{k}(\mathcal{U}_1, \mathcal{U}_2) \neq \emptyset\textup{).}
\]
Moreover, \( \beta^{k}(\mathcal{U}_1, \mathcal{U}_2) < 0 \) if and only if
\begin{equation}\label{eq99}
\mu_1 + \mu_2
<
\lambda_1 \lambda_2
+
\left(
1 -
\frac{n_1^{2} + n_2^{2}}
{k_1\, n_1\, k_2\, n_2}
\right)(g - 1)
-
\frac{2}
{k_1\, n_1\, k_2\, n_2}.
\end{equation}
\end{theorem}

\begin{proof}
By hypothesis, \( S(n, d_2, n_2 + n) \) is non-empty; hence there exists a generated coherent system \((E,V)\) with
\( D_{E,V} \in B(n_2, d_2, k_2) \), where \( k_2 = n + n_2 \).
If \( E_1 \in B(n_1, d_1, k_1) \), then \( H^0(E^* \otimes E_1) = 0 \), since \( \mu(E_1) < \mu(E) \).
It then follows from Lemma~\ref{p3} that \( B^k(\mathcal{U}_1, \mathcal{U}_2) \neq \emptyset \).
The semistable case is handled similarly.
Finally, \eqref{eq99} is obtained from \eqref{eq80} by substituting \( k_2 = n_2 + n \).

\end{proof}

When $E$ is a line bundle (i.e. $n=1$), we have more information about stability and can apply
Lemma \ref{p3}(iii). The result of Mistretta (Theorem \ref{l1}) is already covered, since $D_{L,V}$ then has slope $\le2$ and this case is included in \cite[Theorem 4.3]{bpn}. However, using Theorems \ref{l3} and \ref{t01}, we get the following result.

\begin{cor}\label{tl} Suppose that $C$ is a general curve of genus $g\ge2$. Suppose further that $\beta(1,d_2,n_2+1)\ge0$ and
\[k_2=n_2+1,\ d_1<n_1d_2 \ \mbox{and}  k=k_1k_2.\]
If $B(n_1,d_1,k_1)\ne\emptyset$ (resp., $\widetilde{B}(n_1,d_1,k_1)\ne\emptyset$), then $B^k(\cU_1,\cU_2)\ne\emptyset$ (resp., $\widetilde{B}^k(\cU_1,\cU_2)\ne\emptyset$). Moreover, $\beta^k(\cU_1,\cU_2)<0$ if and only if
\begin{multline}\label{eq0991}\mu_1+\mu_2<\left(1+\frac1{n_1}\right)\lambda_2+\left(1-\frac{n_1^2+n_2^2}{(n_1+1)n_1k_2n_2}\right)(g-1)\\
-\frac2{(n_1+1)n_1k_2n_2}.
\end{multline}
\end{cor}

\begin{proof}
Let $L$ be a line bundle of degree $d_2$ and $V$ a subspace of $H^0(L)$ of dimension $n_2+1$, so that, by Theorems \ref{l3} and \ref{t5}, $S(1,d_2,n_2+1)\ne \emptyset $. If $E_1\in B(n_1,d_1,k_1)$, then $H^0(L^*\otimes E_1)=0$ since $\mu(E_1)<d_2$. It follows from Lemma \ref{p3} and Theorem \ref{t01} that  $B^k(\cU_1,\cU_2)\ne\emptyset$.  \eqref{eq099} is also obtained from \eqref{eq80}.
\end{proof}

\begin{rem}\begin{em}
If $d_2>2n_2$, we avoid overlap with \cite[Theorem 4.3]{bpn}. This happens automatically if $g>n_2+1$ (see \eqref{eq41}). We require the hypothesis that $C$ is general in order to apply Theorem \ref{l3}. In many cases (conjecturally, all), it is sufficient to assume that $C$ is Petri.
\end{em}\end{rem}

\subsection{Examples}\label{ss42}

It is easy to find many examples as in Theorem \ref{tl}, including some for which the BN number is negative. We assume always that $n_i\ge2$ for $i=1,2$.

\begin{ex}\label{ex2}\begin{em}
Suppose that $d_2=n_2+g$, which implies \eqref{eq41}. Given the other hypotheses of Theorem \ref{tl}, this means that $B^k(\cU_1,\cU_2)\ne\emptyset$. In fact, this works for any Petri curve (see \cite[Theorem 4.7]{bbn2} for $n_2\ge3$, and \cite[Theorem 7.1]{bbn1} for $n_2=2$). If $d_2>2n_2$ and $d_1>2n_1$, we avoid overlap with \cite[Theorem 4.3]{bpn} even after taking account of the symmetry $B^k(\cU_2,\cU_1)\cong B^k(\cU_1,\cU_2)$. Moreover, $B(n_2,d_2,n_2+1)\ne\emptyset$. Theorem \ref{tl} now applies provided that $d_1<n_1(n_2+g)$, so that, with $k=(n_2+1)n_1$, $B^k(\cU_1,\cU_2)\ne\emptyset$.
Substituting in \eqref{eq80} and simplifying, we see that $\beta^k(\cU_1,\cU_2)<0$ if and only if
\begin{equation}\label{eq91}
\mu_2<\left(\frac{(n_2^2-2)n_1^2-n_2^2}{n_2(n_2+1)n_1^2}\right)(g-1)-\frac2{n_2(n_2+1)n_1^2}.
\end{equation}
Since the coefficient of $g$ in \eqref{eq91} is strictly positive, this means that, if we fix $n_1$, $n_2$ and $d_2$, then $\beta^k(\cU_1,\cU_2)<0$ for all sufficiently large $g$. Moreover,  \eqref{eq91} gives us the minimum possible value of $g$. For example, if $n_1=n_2=2$, \eqref{eq91} becomes $g>6\mu_2+\frac32$. For the minimum value of $d_2$ which we are allowing,  namely $d_2=2n_2+1$, this becomes $g\ge17$.
\end{em}\end{ex}

\begin{ex}\label{exbpn}\begin{em}
The new BN region (BPN region), introduced in \cite{bpn}, provides new values of \((n_1,d_1,k_1)\) for which \(B(n_1,d_1,k_1) \neq \emptyset\). For example,
for \(g = 10\), we may take \((n_1,d_1,k_1): = (6,22,7)\).
Then $
(\mu, \lambda) = \left(\frac{22}{6}, \frac{7}{6}\right)$
lies in the BPN region and $B(6,22,7) \neq \emptyset.$

For any triplet $(1,d_2,n_2+1)$ with  $d_2\geq d_{\min}:=n_2 + \left\lceil \frac{10n_2}{n_2+1} \right\rceil $, $S(1,d_2,n_2+1)\ne \emptyset $. That $B^k(\mathcal{U}_1, \mathcal{U}_2)\ne \emptyset $ follows from Theorem \ref{tl}, since $$\frac{22}{6}<  n_2 + \left\lceil \frac{10n_2}{n_2+1} \right\rceil \leq d_2.$$
Let $$
d_{\max} = \left\lfloor \frac{264n_2^2 + 322n_2 - 319}{42(n_2+1)} \right\rfloor .
$$
 If $10=g>n_2+1$ and $d_{\min} \leq d_2 \leq d_{\max}$, a straightforward computation shows that $\beta^k(\mathcal{U}_1,\mathcal{U}_2) < 0$.
The table below shows the corresponding ranges of the degree $d_2$ for each possible rank:
\bigskip

$$
\begin{array}{|c|c|}
\hline
&\\
 n_2 & d_{\min}\leq  d_2 \leq d_{\max}  \\
\hline
2 & 9 \le d_2 \le 10 \\
3 & 11 \le d_2 \le 17 \\
4 & 12 \le d_2 \le 24 \\
5 & 14 \le d_2 \le 31 \\
6 & 15 \le d_2 \le 37 \\
7 & 16 \le d_2 \le 44 \\
8 & 17 \le d_2 \le 50 \\
\hline
\end{array}
$$
\smallskip
\centerline{Table 1}

Note that $\beta (2,9,3)=\beta (8,17,9)=1$ and $\beta (4,12,5)=\beta (9,18,10)=0$.
\end{em}\end{ex}

\begin{ex}\label{ex40}\begin{em} Assume $d_i\geq  n_i + \left\lceil \frac{n_i g}{n_i + 1} \right\rceil $ for $i=1,2$. Then, from Theorems \ref{l3} and \ref{t5}, $S(1,d_i,n_i+1)\ne \emptyset .$ In this case the morphism $\Phi$ in  \ref{phi} is defined as
$$(D_{{L_1,V_1}},L_2)\mapsto (D_{{L_1,V_1}},D_{{L_2,V_2}}).$$ If $\frac{d_1}{n_1}<d_2$, then,
from Corollary \ref{tl}, $B^k(\mathcal{U}_1, \mathcal{U}_2)\ne \emptyset $ and \(\beta^k(\mathcal{U}_1, \mathcal{U}_2)<0\) if \ref{eq099} is satisfied.

In particular, take
$$\Phi: B(n_1,d_1,n_1+1)\times S(1,d_1,n_1+1)\lra M(n_1,d_1)\times M(n_1,d_1).$$

Let $k=(n+1)^2$. Since $\frac{d_1}{n_1}<d_1$, $B^k(\mathcal{U}_1, \mathcal{U}_1)\ne \emptyset $ and $\beta^k(\cU_1,\cU_1)<0$ if and only if \begin{multline}\label{eq0999}2\mu_1<\left(1+\frac1{n_1}\right)^2+\left(1-\frac{2n_1^2}{(n_1+1)^2n_1^2}\right)(g-1)\\
-\frac2{(n_1+1)^2n_1^2}.
\end{multline}
Thus, for $k=(n+1)^2$ and
\begin{multline}\label{nnn} n_1 + \left\lceil \frac{n_1 g}{n_1 + 1} \right\rceil \leq d_1< \frac{n_1}{2}\left(1+\frac1{n_1}\right)^2+\frac{n_1}{2}\left(1-\frac{2n_1^2}{(n_1+1)^2n_1^2}\right)(g-1)\\
-\frac{n_1}{(n_1+1)^2n_1^2},
\end{multline}
 $\beta^k(\cU_1,\cU_1)<0$.

 For example, for $g=11$ the triplet $(n_1,d_1,n_1+1)=(3,12,4)$  satisfies the \ref{nnn}. Hence, $S(3,12,4)\ne \emptyset $ and for $k=16$, $B^k(\mathcal{U}_1, \mathcal{U}_1)\ne \emptyset $ and \(\beta^k(\mathcal{U}_1, \mathcal{U}_1)<0\).

\end{em}\end{ex}

\begin{ex}\label{ex4}\begin{em}
For a more refined calculation, we fix $(n_1,d_1,k_1):=(2,5,2)$ and take \(E_1 \in B(2,5,2)\). For \(n_2 \ge 2\), let
\[
d_2 = d_{\min} := n_2 + \left\lceil \frac{n_2 g}{n_2 + 1} \right\rceil .
\]
 Then, \(S(1, d_{\min}, n_2+1)\ne \emptyset \). Again, the non-emptiness of  $B^k(\mathcal{U}_1, \mathcal{U}_2)$ follows from Theorem \ref{tl}, since $$\mu(E_1)=\frac{5}{2} < n_2 + \left\lceil \frac{n_2 g}{n_2 + 1} \right\rceil =d_{\min}. $$
  It is easy to see that no \(n \ge 2\) with \(g \le 4\) satisfies \(\beta^k(\mathcal{U}_1, \mathcal{U}_2)<0\). Moreover,
\[
\beta^{2(n+1)}(\mathcal{U}_1, \mathcal{U}_2) < 0 \quad \Longleftrightarrow \quad g \ge  g_{0},
\]
where \(g_{0}\) denotes the minimal possible value of \(g\) in \eqref{eq099}. For $n_1=8,9,10$, $g_0=5$ and for $n_1=6,7,8$, $g_0$ is $6$. However, to avoid overlap with \cite[Theorem 4.3]{bpn} we assume $g \ge \max\{n+2, g_{0} \}$.
The following table specifies the lower bound of $g$ for each rank between \(2 \le n_2 \le 10\). In that range, new points are obtained with $\beta^{2(n+1)}(\mathcal{U}_1, \mathcal{U}_2) < 0$,
\smallskip
\[
\begin{array}{|c|c|}
\hline
 & \\
 n_2 & g  \\

\hline
2   & g\geq  9  \\
3 & g\geq   7  \\
4 & g\geq   6  \\
5 & g\geq    7   \\
6 & g\geq  8    \\
7 & g\geq  9  \\
8 & g\geq  10  \\
9 &  g\geq   11 \\
10 & g\geq  12 \\
\hline
\end{array}
\]
\centerline{Table 2}
\smallskip

Moreover, $B^k(\mathcal{U}_1, \mathcal{U}_2)\ne \emptyset $ implies that $$\widetilde{B}(2n_2,2d_{min}+5n_2, 2(n_2+1))\ne \emptyset .$$ Let $$\mu _0=\mu _1+\mu _2 =\frac{5}{2}+\frac{d_{min}}{n_2} \ \ \mbox{and} \ \  \lambda _0=\lambda _1\lambda_2= 1+\frac{1}{n_2}.$$ For example, for $g=10$, the points $(\mu _0,\lambda _0)$ with $(n_2,d_{min},n_2+1)$ for $n_2=5,6,7,8$  are outside the BPN and BMNO regions in the BN map, since $\mu _1>2$ and $$\lambda _0> f_{10}(\mu _0 ),$$ where $f_{10}(\mu _0 ),$ is the top boundary of the BMNO region in $\mu _0$ (see \cite{bmno} and \cite{bpn}). Therefore, the points $(\mu _0,\lambda _0)$ define new points in the BN- map. The same conclusion can be obtain for some cases when $d>d_{min}.$ For example if $(n_2,d_2, n_2)$ is one of the  following triples $$(6, 16, 7), (7, 17, 8), (7, 18, 8), (7, 19, 8), (8, 18, 9), (8, 19, 9), (8, 20, 9),(8, 21, 9)$$ the points $(\mu _0,\lambda _0)$ are outside the known region.  We expect that, by using the dual span and $\mu _1>2$, it will be possible to obtain a new region on the BN map.
\end{em}\end{ex}

\begin{ex}\label{ex5}\begin{em}
Using the results of \cite{bmgno}, additional examples can be constructed. Let \(C\) be a general curve of odd genus \(g \ge 5\). By \cite[Theorem~5.2(ii), (iii)]{bmgno}, we have
\[
S(2, g+3, 4) \neq \emptyset.
\]
Taking $(n,d_2,n+2)$ as $(2, g+3, 4)$ and, for example,  \((n_1,d_1,k_1)= (5,12,6)\) with \(g\geq 7\),  Theorem~\ref{tl} gives non-emptiness of
$
B^k(\mathcal{U}_1, \mathcal{U}_2),
$ since $\frac{12}{5}<\frac{g+3}{2}$. Furthermore, substituting into \eqref{eq99} yields
\[
\beta^{24}(\mathcal{U}_1, \mathcal{U}_2) < 0.
\]

 Now take $(n_1,d_1,k_1)$ as $(2, g+3, 4)$ and $S(1,d_2,n_2+1)\ne \emptyset$. In this case,  for any $n_2 \geq 2$, $$\frac{g+3}{2}< n_2 + \left\lceil \frac{n_2 g}{n_2 + 1} \right\rceil \leq d_2 ,$$ hence,
$B^k(\mathcal{U}_1, \mathcal{U}_2) \ne \emptyset  .$ The inequality \eqref{eq99} gives the upper bound $$d_2 <\frac{(n_2+1)}{2}+n_2(g+1)-\frac{n_2}{2}(g+3)-\frac{(4+n_2^2)(g-1)+2}{8(n_2+1)}$$ for $ \beta^k(\mathcal{U}_1, \mathcal{U}_2) < 0. $

\end{em}\end{ex}

\begin{ex}\label{ex61}\begin{em}
From \cite{bmgno} we know that for $g=6$, $S(2,10,5)\ne \emptyset$.  By Theorems \ref{l3} and \ref{t5} one may take,  $B(n_1,d_1,n_1+1)\ne \emptyset $ with
$$  n_2 + \left\lceil \frac{n_2 g}{n_2 + 1} \right\rceil \leq d_1 < 5n_1 .$$ Hence, for any $E_1 \in B(n_1,d_1,k_1)$,  $\mu (E_1)<\frac{10}{2}$ and from Theorem~\ref{tl}
$B^k(\mathcal{U}_1, \mathcal{U}_2) \ne \emptyset $ and $\beta ^k(\mathcal{U}_1, \mathcal{U}_2) <0$. For example, we can take $B(4,10,5)$ or B(2,6,3). In these cases, $g>n_1+1$   and  $B^k(\mathcal{U}_1, \mathcal{U}_2) \ne \emptyset $
with $ \beta^k(\mathcal{U}_1, \mathcal{U}_2) < 0. $
\end{em}\end{ex}

\begin{rem}\begin{em}\label{ra}
Further examples can be obtained using the isomorphisms
 \begin{itemize}
 \item $B^k(\cU_1,\cU_2)\cong B^k(\cU_2,\cU_1),$
 \item $B^k(\cU_1,\cU_2)\cong B^{k_1}(\cU_1^*,\cU_2^*\otimes p_C^*K_C)$.
 \item $B^k(\cU_1,\cU_2)\cong B^k(\cU_1\otimes p_C^*L^*,\cU_2\otimes p_C^*L)$, where $L$ is any line bundle,
 \end{itemize}
However, these isomorphisms have no implications on the BN map, as the  regions BMNO and BPN are, by definition, invariant under Serre duality.
\end{em}\end{rem}

\section{Examples using kernel bundles}\label{kernel}
\subsection{The construction}\label{ss51} In this section, we construct further examples of twisted BN loci,
including some with negative BN number. The argument generalises that of \cite[Section 5]{bpn} and we begin with a
generalisation of \cite[Proposition 5.1]{bpn}. This again forms part of \cite[Remark 5.5]{br}.
In this case we consider the morphism

\begin{equation}\label{eq22}
\Phi: B(n_1,d_1,k_1)\times S(n,d,v)\lra M_1\times M_2:(E_1,E)\mapsto (E_1,D_{E,V}^*),
\end{equation}
and we want to give conditions for $Im \Phi \ne \emptyset$ and $ Im \Phi \subset B^k(\cU_1,\cU_2) .$

\begin{prop}\label{t2}
Suppose that $E_1$ has rank $n_1$ and degree $d_1$ with $h^0(E_1)\ge k_1$ and that $E$ and $D_{E,V}^*$ are as in \eqref{eq6}. If
\begin{equation}\label{eq7}
k\le vk_1 - h^0(E_1\otimes E),
\end{equation}
then $h^0(E_1\otimes D_{E,V}^*)\ge k$.
\end{prop}
\begin{proof} Tensoring \eqref{eq6} by $E_1$ and taking global sections, we obtain
\[h^0(E_1\otimes D_{E,V}^*)\ge v h^0(E_1)-h^0(E_1\otimes E).\]
This gives the result.
\end{proof}

\begin{cor}\label{c1}
Suppose that  $C$ is a smooth curve of genus $g\ge 2$ and there exists $E_1\in B(n_1,d_1,k_1)$ with $h^1(E_1\otimes E)=0$. Suppose further that
\begin{equation}\label{eq62}
k\le vk_1-nd_1-n_1d+nn_1(g-1).
\end{equation}
Let  $(n_2,d_2)=(v-n,-d)$ and suppose that $S(n,d,v)$ is defined as in \eqref{eq60} with $S(n,d,v)\ne\emptyset$. Then the image of the morphism
\begin{equation}\label{eq221}
B(n_1,d_1,k_1)\times S(n,d,v)\lra M_1\times M_2:(E_1,E)\mapsto (E_1,D_{E,V}^*)
\end{equation}
is non-empty and  contained in $B^k(\cU_1,\cU_2)$; in particular,
\[B^k(\cU_1,\cU_2)\ne\emptyset.\]
\end{cor}

\begin{proof}
Since $h^1(E_1\otimes E)=0$, we have $h^0(E_1\otimes E)=nd_1+n_1d-nn_1(g-1)$. The corollary now follows at once from Proposition \ref{t2}.
\end{proof}

We know many cases where $B(n_1,d_1,k_1)\ne\emptyset$ (see \cite[Section 2]{bpn} and Subsection \ref{ss23}). In order to generate examples from Corollary \ref{c1}, we need to know that the right-hand side of \eqref{eq62} is positive and that $S(n,d,v)\ne\emptyset$. The former is easy to arrange; for example we could take $k_1>n_1$, $V$ of small codimension and $d$ sufficiently large (we will be more precise in Subsection \ref{ss52}). For the latter, in addition to the examples covered by \cite[Theorem 5.3]{bpn}, we have many examples with $n=1$ (see Theorem \ref{l3}).

\begin{rem}\label{r95}\begin{em}
If there does not exist $E_1\in B(n_1,d_1,k_1)$ with $h^1(E_1\otimes E)=0$, we could define $m$ to be the minimum value of $h^1(E_1\otimes E)$ for such $E_1$ and replace the left-hand side of \eqref{eq62} by $k+m$. Note also that $h^1(E_1\otimes E)=0$ for sufficiently large $d$.
\end{em}\end{rem}

\begin{theorem}\label{t4} Suppose that  $C$ is a smooth curve of genus $g\ge 2$, $n_1\ge2$, $k_1>n_1$ and there exists $E_1\in B(n_1,d_1,k_1)$ with $h^1(E_1\otimes E)=0$. Suppose further that  \eqref{eq62} holds. Let  $(n_2,d_2)=(v-n,-d)$ and suppose that $S(n,d,v)$ is defined as in \eqref{eq60} with $S(n,d,v)\ne\emptyset$. If, in addition, $k=d(k_1-n_1)-e$ and $v=d-n(g-1)-f$, where
\begin{equation}\label{eq70}
 n(g-1)(k_1-n_1)+nd_1+fk_1\le e<d(k_1-n_1),
\end{equation}
then,  for any fixed values of $n_1$, $d_1$, $k_1$, $n$, $e$ and $f$ satisfying \eqref{eq70}, and $d\gg0$, $B^k(\cU_1,\cU_2)\ne\emptyset$. Moreover, if
\begin{equation}\label{eq34}
 d_1<k_1+n_1(g-1)-\frac{g-1}{k_1-n_1},
\end{equation}
then
\begin{equation}\label{eq25}\beta^k(\cU_1,\cU_2)<0.
\end{equation}
\end{theorem}
\begin{proof}  Note that \eqref{eq70} ensures that we can have $k\ge1$ in \eqref{eq62}. It now follows directly from Corollary \ref{c1} that $B^k(\cU_1,\cU_2)\ne\emptyset$ .
The rest is proved exactly as for \cite[Theorem 5.3]{bpn}. The formula for $\beta^k(\cU_1,\cU_2)$ is more complicated, but it remains a quadratic in $d$ with leading coefficient
\begin{equation}\label{eq77}
g-1-(k_1-n_1)(k_1-n_1-d_1+n_1g).
\end{equation}
 So $\beta^k(\cU_1,\cU_2)<0$ for $d\gg0$ provided that \eqref{eq34} holds.
\end{proof}

\begin{rem}\label{r4}\begin{em}
Since $v>n$, we have $f<d-ng$. Note that it is possible that $f$ could be negative, but not under the assumption $d\gg0$.
\end{em}\end{rem}

\subsection{Examples}\label{ss52}

Let $(E,V)\in S(n,d,v)$ and $E_1\in B(n_1,d_1,k_1)$
Tensoring \eqref{eq6} by $E_1$ and taking global sections, we obtain the following exact sequence
$$0\to H^0(E_1\otimes D_{E,V}^*) \to V\otimes H^0(E_1)\to H^0(E\otimes E_1)\to \dots .$$
\begin{rem}\label{exf}\begin{em}
To show that Theorem \ref{t4} provides the examples we are seeking, we need to show that the following conditions are compatible:
\begin{enumerate}
\item $\mu (E_1\otimes D_{E,V}^*)>0$,
\item $h^1(E_1\otimes E)=0$,
\item $vk-h^0(E_1\otimes E)>0$ and
\item $((v-n)^2+n_1^2)(g-1)+2 -vk_1(vk_1-(d(E_1\otimes D_{E,V}^*)+(v-n)n_1(g-1)<0.$
\end{enumerate}
\end{em}\end{rem}
Condition $2)$ follows from the inequality $\mu (E_1\otimes E)>2g-2$.

\begin{ex}\label{ex11}\begin{em}
Let  $C$ be a general curve of genus $g = 6$. From \cite{bmgno} we know that $B(2,10,5)$ is non-empty. Let $E_1\in B(2,10,5)$ and $(L,V)\in S(1,10,5).$ In this case,
\begin{enumerate}
\item $\mu (E_1\otimes D_{L,V}^*)= 5-\frac{10}{4}>0.$
\item $\mu (E_1\otimes L)= 5+ 10 > 10=2g-2$. Therefore, $h^1(E_1\otimes L)=0$.
\item $h^0(E_1\otimes L)= d(E_1\otimes L)+2\cdot (1-g)= 20 $, hence $vk-h^0(E_1\otimes E)=25-20=5>0$.
\end{enumerate}
 Thus, $B^5(\cU_1,\cU_2)\ne \emptyset $ and a straightforward calculation gives $\beta^k(\cU_1,\cU_2)<0$.
 \end{em}\end{ex}

 In the case $k_1=n_1+1$ we possess more precise statements concerning the possible values of $d_1$.

\begin{cor}\label{t3}
Let  $C$ be a Petri curve of genus $g\ge 3$ ,$n_1\ge2$, $k_1=n_1+1$ and $(n_2,d_2)=(d-ng,-d)$ with $d\ge 2ng$.  Write $e:=d-k$ and let $(n_2,d_2)=(d-ng,-d)$. Suppose further that
\begin{equation}\label{eq71}
e\ge n(g-1)+nd_1
\end{equation}
and
\begin{equation}\label{eq8}
 n_1+\frac{n_1g}{n_1+1}\le d_1<n_1+g+\frac{(n_1^2-n_1-2)g}{2(n_1-1)^2}.
\end{equation}
Then, for any fixed value of $e$ satisfying \ref{eq71} and $d\gg0$, $B^k(\cU_1,\cU_2)\ne\emptyset$ and $\beta^k(\cU_1,\cU_2)<0$. In particular, $B^k(\cU_1,\cU_2)$ is of negative expected dimension and its non-emptiness cannot be obtained from Theorem \ref{t4}.
\end{cor}
\begin{proof} It follows from \eqref{eq8} and \cite[Theorem 4.7]{bbn2} that $B(n_1,d_1,n_1+1)\ne\emptyset$. Now note that
\[n_1+g+\frac{(n_1^2-n_1-2)g}{2(n_1-1)^2}\le(n_1-1)g+2.\] The result follows from \eqref{eq71} and Theorem \ref{t4}.
\end{proof}

\begin{ex}\label{ex6}\begin{em}If $n_1=2$, the second inequality in \eqref{eq8} is not needed \cite[Theorem 7.1]{bbn1}. So $d_1$ can be any integer in the range
\begin{equation}\label{eq74}
2+\left\lceil\frac{2g}3\right\rceil\le d_1\le g+1.
\end{equation}
The cases $n_1=3$ and $n_1=4$ are similar \cite[Theorems 7.2, 7.3]{bbn1} and the range for $d_1$ is
\begin{equation}\label{eq72}
n_1+\left\lceil\frac{n_1g}{n_1+1}\right\rceil\le d_1\le(n_1-1)g+1.
\end{equation}
For $n_1\ge5$, we certainly have $n_1^2-n_1-2>(n_1-1)^2$, so all values in the range
\[n_1+\left\lceil\frac{n_1g}{n_1+1}\right\rceil\le d_1\le n_1+\left\lfloor\frac{3g}2\right\rfloor\]
are possible. If $g\ge2n_1-4$, then the full range \eqref{eq72} is permitted (see \cite[Theorem 6.1]{bbn2}.
\end{em}\end{ex}

\begin{rem}\label{r5}\begin{em}
(i) For general $C$, Theorem \ref{l3} implies that Corollary \ref{t3} is valid with \eqref{eq72} replacing \eqref{eq8}.

(ii) When $g=2$, we know that $B(n_1,d_1,n_1+1)\ne\emptyset$ if and only if $d_1\ge n_1+2$ and $d_1\ne2n_1$ (see \cite[Theorem 8.2]{bbn1}. The conclusions of Corollary \ref{t3} are therefore valid and yield examples provided $n_1\ge3$. If $n_1=2$, \eqref{eq74} cannot be satisfied and Corollary \ref{t3} yields no examples.
\end{em}\end{rem}

\subsection{Examples with $n=1$}\label{ss53}
The results of the previous subsections are valid, in particular,  when $n=1$. However, it is also possible to consider generated linear systems $(L,V)$, since we know a good deal about the stability of $D^*_{L,V}$ (see Theorems \ref{l1} and \ref{l3}). In the first place, we have the following analogue of Theorem \ref{t4}.

\begin{theorem}\label{t10}
Suppose that $C$ is a smooth curve of genus $g\ge2$, $n_1\ge2$, $k_1>n_1$, $c$ an integer, $1\le c\le g$, and either $d>2g+2c$ or $d=2g+2c$ and $C$ is non-hyperelliptic. Suppose further that $B(n_1,d_1,k_1)\ne\emptyset$ and that
\begin{equation}\label{eq75}
k\le(d-g+1)(k_1-n_1)-ck_1-d_1.
\end{equation}
Let $(n_2,d_2)=(d-g-c,-d)$. Then $B^k(\cU_1,\cU_2)\ne\emptyset$.

If, in addition, $k=d(k_1-n_1)-e$, where
\begin{equation}\label{eq78}
e\ge (g-1)(k_1-n_1)+ck_1+d_1
\end{equation} and \eqref{eq34} holds, then, for any fixed values of $n_1$, $d_1$, $k_1$, $c$ and $e$, and $d\gg0$, $\beta^k(\cU_1,\cU_2)<0$.
\end{theorem}

\begin{proof}
Let $E_1\in B(n_1,d_1,k_1)$ and let $(L,V)$ be a general linear system of type $(d,d-g+1-c)$. By Theorem \ref{l1}, $(L,V)$ is generated and $D_{L,V}\in B(n_2,d_2,n_2+1)$. Since $E_1\otimes L$ is stable of slope $>2g-2$, $h^1(E_1\otimes L)=0$ and $h^0(E_1\otimes L)=d_1+n_1d-n_1(g-1)$. So, by Proposition \ref{t2},
\begin{eqnarray}\label{eq76}
\nonumber h^0(E_1\otimes D^*_{L,V})&\ge& (d-g+1-c)k_1-(d_1+n_1d-n_1(g-1))\\
&=&(d-g+1)(k_1-n_1)-ck_1-d_1.
\end{eqnarray}
It follows that, if \eqref{eq75} holds, then $B^k(\cU_1,\cU_2)\ne\emptyset$.

For the second part, we deduce that, for fixed values of $n_1$, $d_1$, $k_1$, $c$ and $e$, $\beta^k(\cU_1,\cU_2)$ is again a quadratic in $d$ with leading coefficient given by \eqref{eq77}. It follows that $\beta^k(\cU_1,\cU_2)<0$ for $d\gg0$ if \eqref{eq34} holds.
\end{proof}

The following corollary follows from Theorem \ref{t10} in exactly the same way as Corollary \ref{t3} follows from Theorem \ref{t4}.

\begin{cor}\label{c8}
Suppose that  $C$ is a smooth curve of genus $g\ge 3$, $n_1\ge2$, $k_1>n_1$, $1\le c\le g$ and $(n_2,d_2)=(d-g-c,-d)$ with either $d>2g+2c$ or $d=2g+2c$ and $C$ non-hyperelliptic.
\begin{itemize}
\item[(i)] If  \eqref{eq62} holds and $d_1$ is not divisible by $n_1$, then $B^k(\cU_1,\cU_2)\ne\emptyset$. Moreover, for
any fixed value of $e$ satisfying \eqref{eq78} and $d\gg0$, $\beta^k(\cU_1,\cU_2)<0$.
\item[(ii)] If $C$ is Petri, $k_1=n_1+1$ and \eqref{eq8} holds, then $B^k(\cU_1,\cU_2)\ne\emptyset$. Moreover, for
any fixed value of $e$ satisfying
\[e\ge g-1+c(n_1+1)+d_1\]
and $d\gg0$, $\beta^k(\cU_1,\cU_2)<0$.
\end{itemize}
\end{cor}

This corollary implies that examples of twisted BN loci with negative expected dimension can be constructed exactly as in Example \ref{ex6}.

\begin{ex}\label{ex8}\begin{em} Let $a,b$ and $c$, be positive integers.  Assume $d=2g+2c+a$ with $0\leq c \leq g$.   Let $(L,V)$ be a linear system of type $(d, g+c+a+1)$, i.e.
$$\upsilon :=\dim V= h^0(L)-c=g+c+a+1.$$ From Theorem \ref{l1}, $D_{L,V}$ is stable of rank $g+c+a$.  Let $E_1\in B(n_1,d_1, k_1)$ with $k_1:=n_1+b$ and $2<\frac{d_1}{n_1}.$ Hence, $1)$ and $2)$ of \ref{exf} are satisfied and $$k:=\upsilon \cdot k_1-h^0(L\otimes E_1)= b(g+c+a+1)-d_1-nc.$$
Certainly we can give values to $c$ and $a$ to obtain $k>0$.

Let $g=6$, $c=4$ and $a=6$ and $d=2g+2c+a=26$. Let $(L,V)$ be a generated linear system of type $(d, g+c+a+1)= (26,17)$. From Theorem \ref{l1}, $D_{L,V}$ is stable of rank $g+c+a= 16$, hence, $B(16,26,17) \ne \emptyset$. Take now $E_1\in B(2,10,5)$. In this case
$$\upsilon \cdot k_1-h^0(L\otimes E_1)= b(g+c+a+1)-d_1-nc = 33.$$ Since $3)$ and $4)$ of \ref{exf} are also satisfied, $B^{33}(\cU_1,\cU_2)\ne\emptyset$ and $\beta^{33}(\cU_1,\cU_2)<0$.
\end{em}\end{ex}

We finish with a further example based on \cite[Theorem 1.3]{fla}.
\begin{ex}\label{ex81}\begin{em}
Let $C$ be a general curve of genus $g\ge2$ and suppose that $n_1\ge2$. If $g=2$, assume that $d\ne2(v-1)$. According to Theorem \ref{l3}, if $v\ge3$, there exists a generated linear system $(L,V)$ of type $(d,v)$ such that $D_{L,V}$ is stable if and only if $d\ge v-1+\frac{(v-1)g}v$. Write $v=d-g+1-c$, where we no longer assume that $1\le c\le g$ (indeed, $c$ could be negative). The proof of Theorem \ref{t10} now works provided that $h^1(E_1\otimes L)=0$, which is certainly true if $d_1+n_1d>n_1(2g-2)$. Hence, if we fix $n_1$, $d_1$, $k_1$, $c$ and $e$ such that \eqref{eq78} holds, then $B^k(\cU_1,\cU_2)\ne\emptyset$ and $\beta^k(\cU_1,\cU_2)<0$ for $d\gg0$. The constructions of Example \ref{ex6} now yields many twisted BN loci of negative expected dimension.
\end{em}\end{ex}

\end{document}